\documentclass[reqno]{amsart}

\usepackage[latin1]{inputenc} 
\usepackage{latexsym,amsfonts,amssymb,amsmath,euscript}

\usepackage{graphicx} 
\usepackage{color,fancybox}
\usepackage{color}
\usepackage{graphicx} 

\newcommand{\R}
 {\mathbb{R}}

\begin{document}

\title[Correctors for the Neumann problem in thin domains]{Correctors for the Neumann problem in thin domains with locally periodic oscillatory structure}

\author[M. C. Pereira]{Marcone C. Pereira$^\dag$}
\thanks{$^\dag$Partially
supported by CNPq 305210/2008-4 and FAPESP 2008/53094-4 and 2010/18790-0, Brazil}
\address[Marcone C. Pereira]{Escola de Artes, Ci\^encias e Humanidades,
Universidade de S\~ao Paulo, S\~ao Paulo SP, Brazil}
\email{marcone@usp.br}

\author[R. P. Silva]{Ricardo P. Silva$^\star$}
\thanks{$^\star$Partially
supported by FUNDUNESP 0135812 and FAPESP 2008/53094-4 and 2012/06753-8, Brazil}
\address[Ricardo P. Silva]{Instituto de Geoci\^{e}ncias e Ci\^{e}ncias Exatas, Universidade Estadual Paulista, Rio Claro SP, Brazil}
\email{rpsilva@rc.unesp.br}

\date{}


\begin{abstract}
In this paper we are concerned with convergence of solutions of the Poisson equation with Neumann boun\-da\-ry conditions in a two-dimensional thin domain exhibiting highly oscillatory behavior in part of its boun\-dary. We deal with the resonant case in which the height, amplitude and period of the oscillations are all of the same order  which is given by a small parameter $\epsilon > 0$. 
Applying an appropriate corrector approach we get strong convergence when we replace the original solutions by a kind of first-order expansion through the Multiple-Scale Method.
\end{abstract}

\maketitle
\numberwithin{equation}{section}
\newtheorem{theorem}{Theorem}[section]
\newtheorem{lemma}[theorem]{Lemma}
\newtheorem{corollary}[theorem]{Corollary}
\newtheorem{proposition}[theorem]{Proposition}
\newtheorem{remark}[theorem]{Remark}
\allowdisplaybreaks

\section{Introduction}

Given a small positive parameter $\epsilon$, our goal in this work is to analyze the convergence of the family of solutions $w^\epsilon \in H^{1}(R^{\epsilon})$ of the 
Neumann boundary problem
%
\begin{equation} \label{P}
\left\{
\begin{gathered}
- \Delta w^\epsilon + w^\epsilon = f^\epsilon
\quad \textrm{ in } R^\epsilon, \\
\frac{\partial w^\epsilon}{\partial \nu ^\epsilon} = 0
\quad \textrm{ on } \partial R^\epsilon,
\end{gathered}
\right. 
\end{equation}
where $f^\epsilon \in L^2(R^\epsilon)$ is uniformly bounded which respect to $\epsilon$, $\nu^\epsilon = (\nu^\epsilon_1, \nu^\epsilon_2)$ denotes the unit outward normal vector field to $\partial R^\epsilon$ and $\frac{\partial }{\partial \nu^\epsilon}$ denotes the outside normal derivative. The domain $R^\epsilon \subset \R^{2}$ is a thin domain with a highly oscillating boundary which is given by
\begin{equation} \label{TD}
R^\epsilon = \{ (x_1,x_2) \in \R^2 \, : \,  x_1 \in (0,1),  \, -\epsilon \, b(x_1) < x_2 < \epsilon \, G_\epsilon(x_1) \},
\end{equation}
where $b$ and $G_\epsilon$ are smooth functions, uniformly bounded by positive constants $b_1$, $G_0$, and $G_1$, independent of $\epsilon$, satisfying $0 \leq b(x) \leq b_1$ and $0<G_0\leq G_\epsilon(x)\leq G_1$ for all $x \in [0,1]$. 
Noticing  that $R^\epsilon \subset (0,1) \times (- \epsilon \, b_1, \epsilon \, G_1)$ for all $\epsilon>0$, we see that the domain $R^\epsilon$ is actually thin on the $x_{2}$-direction. Therefore we will say that $R^\epsilon$ is collapsing to the interval $(0,1)$ as $\epsilon \to 0$.

Here, the function $b$, independent on $\epsilon$, establishes the lower boundary of the thin domain, and the function $G_\epsilon$, dependent on the parameter $\epsilon$, defines its upper boundary. Also, allowing $G_\epsilon$ to present oscillatory behavior we obtain an oscillating thin domain $R^\epsilon$, whose amplitude and period of the oscillations are of the same order $\epsilon$, which also coincides with the order of its thickness.
For instance, we may think the function $G_\epsilon$ of the form $G_\epsilon(x)=m(x)+k(x)g(x/\epsilon)$, where $m, k:[0,1] \to \R$ are ${C}^1$-functions and $g:\R\to \R$ is a $L$-periodic and smooth function. 
In this case, it is clear that the functions $m$ and $b$ define the profile of the thin domain, as well as $k$ the amplitude of the oscillations which is given by the periodic function $g$. This includes the case where the function $G_\epsilon$ is a purely periodic function, as $G_\epsilon(x)=2+\sin(x/\epsilon)$, but also includes the case where the function $G_\epsilon$ exhibit a profile and amplitude of oscillations dependent on $x \in (0,1)$. We can refer the reader to \cite{CFP1} for an example of this kind of locally periodic structure. In fact, we are able to treat even more general cases, see hypotheses $(H1)$ and $(H2)$ in the Section \ref{PRE}.
Figure 1 below illustrates an example of such a thin domain that we are able to treat in this work.

\begin{figure}[htp]
\centering \scalebox{1.0}{\includegraphics{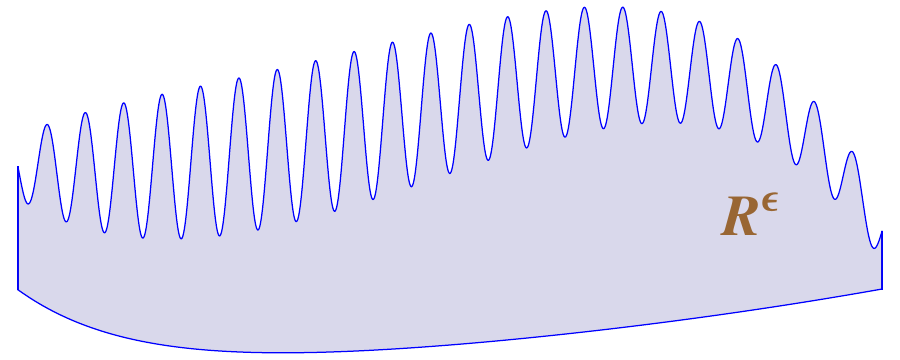}}
\label{thin-domain}
\caption{An oscillating thin domain $R^\epsilon$.}
\end{figure}

Notice that the existence and uniqueness of solutions for problem (\ref{P}) is guaranteed by Lax-Milgram Theorem. The aim of this work is to treat the convergence aspects of the solution $w^\epsilon$ as $\epsilon\to 0$, to an effective limit solution as the domain gets thinner and thinner although the oscillations becomes higher.

In \cite{ACPS,AP,AP2}, the authors combine methods from \emph{homogenization theory} and \emph{boundary perturbation problems} for partial differential equations to obtain the \emph{homogenized (effective) equation} associated to \eqref{P}. 
Since the thin domain $R^\epsilon$ degenerates to the interval $(0,1)$ as $\epsilon$ goes to zero, an one-dimensional limiting equation is established. The respective solution, $w_0 \in H^1(0,1)$, is the limit of the family $w^\epsilon \in H^1(R^\epsilon)$. 
We also have that the \emph{homogenized coefficient} of the limit equation depends on a family of auxiliary solutions, capturing somehow, the oscillatory geometry of the thin domain.
We refer the reader to \cite{BLP,CP, SP, Tt} for a general introduction to the theory of homogenization and to \cite{Henry} for a general treatise on boundary perturbation problems for partial differential equations. 

It is important to notice that the convergence of the sequence $w^\epsilon \in H^1(R^\epsilon)$ to the solution $w_0 \in H^1(0,1)$ of the homogenized equation established in \cite[Theorem 2.3]{AP2} is only weak with respect to $H^1$-norms. In fact, weak convergence in $H^1$-norms is the best one for these kind of problems exhibiting oscillatory behavior since in \cite{A, BLP, CD}, the authors show that in general strong convergence is not possible. 
For this reason, seeking for better convergence results, we apply an appropriate \emph{corrector} approach developed by Bensoussan, Lions and Papanicolaou in \cite{BLP}, to obtain a sort of strong convergence. The ideia is adjust the homogenized solution $w_0$, adding suitable \emph{corrector functions} $\kappa^\epsilon \in H^1(R^\epsilon)$, $\kappa^\epsilon=\text{o}(\epsilon)$ in $L^2(R^\epsilon)$, obtaining 
\begin{equation*} \label{SC}
\epsilon^{-1/2} \|w^\epsilon - w_0 - \kappa^\epsilon \|_{H^1(R^\epsilon)} \to 0, \quad \textrm{ as } \epsilon \to 0. 
\end{equation*}
%
%
Here we proceed as in \cite{HR, PR01, R} rescaling the Lebesgue measure of $H^1(R^\epsilon)$ by a factor $\epsilon^{-1}$ in order to preserve the \emph{relative capacity} of mensurable subset  $\mathcal{O} \subset R^{\epsilon}$ for small values of $\epsilon$. 
Our challenge here is to set a proper family of corrector functions in the way to get the desired strong convergence once the auxiliary solutions associated with the homogenized (limiting) equation are posed in an one-parameter family of \emph{representative cells}.
See Theorem \ref{theo:first-order} for a complete statement on these results.

In such problems, where the analytical solution demands a substantial effort or is unknown, it is of interest in discussing error estimates (in norm) when we replace it by numerical approximations. In this sense, the analysis performed here is fundamental in order understanding the convergence properties of the solutions of the problem \eqref{P}.

There are several works in the literature dealing with boundary value problems posed in thin domains presenting oscillating behavior. Among others, we mention \cite{MV, MP, PS} where asymptotic approximations of solutions  to elliptic problems in thin perforated domains defined by purely periodic functions with rapidly varying thickness were studied, as well as \cite{BGG, BGM}, dealing with nonlinear monotone problems in a multidomain with a highly oscillating boundary. 
Let us also cite \cite{AB, BZ} where the asymptotic description of nonlinearly elastic thin films with fast-oscillating profile was successfully obtained in a context of \emph{$\Gamma$-convergence}  \cite{Dal}.

We also can mention other works addressing the problem of the behavior of solutions of partial differential equations where just part of the boundary of the domain presents oscillation.   
This is the case in \cite{CFP} where the authors deal with the Poisson equation with Robin type boundary conditions in the oscillating part of the boundary.   
In general, the differential equation is not affected by the perturbation but the boundary condition. Some very interesting interplay among the geometry of the oscillations and  Robin boundary conditions can be analyzed and depending on this balance, the limiting boundary condition may be of a different type.  
Similar discussion can be found in \cite{CLS} for the Navier-Stokes equations with slip boundary conditions and in \cite{ArrBru10} for nonlinear elliptic equations with nonlinear boundary conditions as well the references given there.  
In the case of the present paper, the effect of the oscillations at the boundary is coupled with the effect of the domain being thin. As previously observed, the oscillatory behavior considered here affect completely the limit behavior of the equation. 

Let us also point out that thin structures with rough contours as thin rods,  plates or shells, as well as fluids filling out thin domains in lubrication procedures, or even chemical diffusion process in the presence of grainy narrow strips are very common in engineering and applied sciences. 
The analysis of the properties of these processes taking place on them, and understanding how the micro geometry of the thin structure affects the macro properties of the material is a very relevant issue in engineering and material design.  

Therefore,  being able to obtain the limiting equation of a prototype equation like the Poisson equation \eqref{P} in different structures where the micro geometry is not necessarily smooth and being able to analyze how the different micro scales affects the limiting problem, goes in this direction and will allow the research and deep understanding in more complicated situations.    
In this context we also mention \cite{Apl3, Apl1, Apl4} for some concrete applied problems. 

We also would like to observe that although we just treat the Poisson equation with Neumann boundary conditions, we also may consider other different conditions on the lateral boundaries of the thin domain while preserving the Neumann type boundary condition in the upper and lower boundary. Indeed, the limiting problem preserves the boundary condition.

The structure of this paper is as follows. In Section \ref{PRE}, we establish the functional setting for the perturbed equation \eqref{P} and we introduce its limiting problem, and in Section \ref{SFC} we obtain convergence in $H^1$-norms using the corrector approach. In the appendix we show smoothness of the auxiliary solutions with respect to the perturbation on the representative cell.



\section{Preliminaries} \label{PRE}

In the following, $\epsilon$ will denote a small positive parameter which will converges to $0$. In order to setup the thin domain \eqref{TD} we consider a function $b: [0,1] \to [0,\infty)$ and a family of positive functions $G_\epsilon:[0,1]\to (0,\infty)$, satisfying the following hypothesis

\begin{itemize}
\item[(H1)] 
The function $b$ belongs to ${C}^1(0,1)$, it is continuous in $[0,1]$ and satisfies $0 \leq b(x) \leq b_1$ for some positive constant $b_1$ and all $x \in [0,1]$;
\item[(H2)] There exist two positive constants $G_0$, $G_1$ such that 
\begin{equation} \label{HG}
\begin{gathered}
0< G_0 \le G_\epsilon(x) \le G_1, \quad  \forall \ x \in  [0,1], \quad \forall \ \epsilon \in (0,\epsilon_0).
\end{gathered}
\end{equation}
Furthermore, the  functions $G_\epsilon(\cdot)$ have the following structure $G_\epsilon(x)=G(x,x/\epsilon)$, where  the function
\begin{equation}\label{def-G}
\begin{array}{rcl}
G:[0,1] \times \R &\longrightarrow & (0,\infty) \\
 (x,y)& \mapsto & G(x,y)
 \end{array}
 \end{equation}
 is of class ${C}^1$ in $(0,1)\times \R$, continuous in $[0,1]\times \R$, $L$-periodic in the second variable $y$, i.e., $G(x,y+L)=G(x,y)$ for all $(x,y) \in [0,1]\times \R$. 

\end{itemize}

%
%
%

We stress for the fact that $R^\epsilon$ varies in accordance with the positive parameter $\epsilon$ and, when $\epsilon$ goes to $0$, the domains $R^\epsilon$ collapse themselves to the unit interval of the real line, therefore, in order to preserve the relative capacity of a mensurable subset  $\mathcal{O} \subset R^{\epsilon}$  we rescale the Lebesgue measure by a factor ${1}/{\epsilon}$, i.e., we deal with the singular measure 
$$
\rho_{\epsilon}(\mathcal{O})=\epsilon^{-1} | \mathcal{O} |,
$$ 
widely considered in thin domains problems (see for instance \cite{HR,PR01,R}). With this measure we introduce the Lebesgue and Sobolev spaces $L^2(R^\epsilon; \rho_{\epsilon})$ and $H^1(R^\epsilon;\rho_{\epsilon})$. The norms in these spaces will be denoted, respectivelly, by $||| \cdot |||_{L^2(R^\epsilon)}$ and $||| \cdot |||_{H^1(R^\epsilon)}$, which are induced by the inner products
\begin{equation*}
(u,v)_\epsilon = \epsilon^{-1} \int_{R^\epsilon} u \, v \, dx_1 dx_2, \quad \forall \, u,v \in L^2(R^\epsilon)
\end{equation*} 
and
\begin{equation*}
a_\epsilon(u,v) = \epsilon^{-1} \int_{R^\epsilon} \{ \nabla u \cdot \nabla v + u \, v \} dx_1 dx_2, \quad \forall \, u,v \in H^1(R^\epsilon).
\end{equation*} 

\begin{remark}
The $||| \cdot |||$ - norms and the usual ones in $L^2(R^\epsilon)$ and $H^1(R^\epsilon)$ are equivalents and easily related 
$$
\begin{gathered}
||| u |||_{L^2(R^\epsilon)} = \epsilon^{-1/2} \| u \|_{L^2(R^\epsilon)}, \quad \forall \, u \in L^2(R^\epsilon) \\
||| u |||_{H^1(R^\epsilon)} = \epsilon^{-1/2} \| u \|_{H^1(R^\epsilon)}, \quad \forall \,u \in H^1(R^\epsilon).
\end{gathered}
$$
\end{remark}

The variational formulation of (\ref{P}) is: find $w^\epsilon \in H^1(R^\epsilon)$ such that 
\begin{equation} \label{VFP}
\int_{R^\epsilon} \Big\{ \nabla w^\epsilon \cdot \nabla \varphi 
+ w^\epsilon \, \varphi \Big\} dx_1 dx_2 = \int_{R^\epsilon} f^\epsilon \, \varphi \, dx_1 dx_2, \quad \forall \, \varphi \in H^1(R^\epsilon),
\end{equation}
which is equivalent to find $w^\epsilon \in H^1(R^\epsilon;\rho^\epsilon)$ such that 
\begin{equation}\label{eq:vfor}
a_\epsilon(\varphi, w^\epsilon) = (\varphi, f^\epsilon)_\epsilon,  \quad \forall \, \varphi \in H^1(R^\epsilon;\rho_\epsilon). 
\end{equation}

We also observe that the solutions $w^\epsilon$ satisfy a priori estimates uniformly in $\epsilon$. In fact, taking $\varphi = w^\epsilon$ as a test function we derive from (\ref{VFP}) and (\ref{eq:vfor}) that
\begin{equation} \label{priori}
\begin{gathered}
\| \nabla w^\epsilon \|_{L^2(R^\epsilon)}^2
+ \| w^\epsilon \|_{L^2(R^\epsilon)}^2
\le \| f^\epsilon \|_{L^2(R^\epsilon)} \| w^\epsilon \|_{L^2(R^\epsilon)} \\
||| \nabla w^\epsilon |||_{L^2(R^\epsilon)}^2
+ ||| w^\epsilon |||_{L^2(R^\epsilon)}^2
\le ||| f^\epsilon |||_{L^2(R^\epsilon)} ||| w^\epsilon |||_{L^2(R^\epsilon)}.
\end{gathered}
\end{equation}

Now, in order to capture the limiting behavior of $a_\epsilon(w^\epsilon, w^\epsilon)$, as $\epsilon \to 0$, we consider the sesquilinear form $a_0$ in $H^1(0,1)$ defined by
\begin{equation} \label{IPA0}
a_0(u,v)  = \int_0^1  \left\{ r \, u_x \, v_x +  p \, u \, v \right\} dx, \quad \forall \, u,v \in H^1(0,1),
\end{equation}
where $r:(0,1) \to \R$ and $p:(0,1) \to \R$, defined in \eqref{HC}, are positive and bounded functions.
We also consider the following inner product $(\cdot, \cdot)_0$ in $L^2(0,1)$ given by
\begin{equation} \label{IP0}
(u,v)_0 = \int_{0}^1 p\, u  \, v\, dx, \quad \forall \, u,v \in L^2(0,1).
\end{equation}

\begin{remark}
There exists a relationship among the function $p$, the period $L$ of the function $G(x,\cdot)$ and the measure of the representative cell, $Y^*(x)$ $($defined in \eqref{CELLL}$)$, by the expression
\begin{equation} \label{LpY}
L \, p(x) = |Y^*(x)|, \quad \forall x \in (0,1).
\end{equation}
\end{remark}

With such considerations, one can show that the \emph{homogenized problem} associated to \eqref{P} is the Neumann problem
\begin{equation} \label{HL}
\int_0^1 \Big\{ r \, w_x \, \varphi_x
+ p \, w \, \varphi \Big\} dx = \int_0^1  \, \hat f \, \varphi \, dx, \quad \forall \varphi \in H^1(0,1),
\end{equation}
where $r$, $p:(0,1) \to \R^+$ are positive functions defined by
\begin{equation} \label{HC}
\begin{gathered}
r(x) =  \frac{1}{L}\int_{Y^*(x)}\Big\{ 1 - \partial_y X(x)(y,z) \Big\} dy dz, \\ 
p(x) = \frac{|Y^*(x)|}{L}, 
\end{gathered}
\end{equation} 
$X(x)$ is the unique solution of the auxiliary problem 
\begin{equation} \label{AUXP}
\left\{
\begin{gathered}
- \Delta X(x)  =  0  \textrm{ in } Y^*(x)  \\
\frac{\partial X(x)}{\partial N}(y,g(y))= N_1(x) \quad  \textrm{ on } B_1(x)\\
\frac{\partial X(x)}{\partial N}(y,z)=0 \quad \textrm{ on } B_2(x)\\
 X(x)(\cdot, z) \quad L\text{ - periodic} \\ 
\int_{Y^*(x)} X(x) \; dy dz  =  0  
\end{gathered}
\right.
\end{equation}
in the representative cell $Y^*(x)$ given by
\begin{equation} \label{CELLL}
Y^*(x) = \{ (y,z) \in \R^2 \; : \; 0< y < L, \   - b(x) < z < G(x,y) \},
\end{equation}
where $B_1(x)$ is the upper boundary, $B_2(x)$ is the lower boundary of $\partial Y^*(x)$ for $x \in (0,1)$ and $N(x)=(N_1(x),N_2(x))$ is the unitary outside vector field to $\partial Y^*(x)$ .
The function $\hat f \in L^2(0,1)$ is such that the sequence defined by 
$$
\hat f^\epsilon(\xi) = \epsilon^{-1} \int_{-\epsilon b(\xi)}^{\epsilon G_\epsilon(\xi)} f^\epsilon (\xi,s) \, ds, \quad \xi \in (0,1),
$$ 
satisfies $\hat f^\epsilon \stackrel{\epsilon \to 0}{\rightharpoonup} \hat f$, $w-L^2(0,1)$.
Indeed, it follows from \cite[Theorem 2.3]{AP2} that if $w_0 \in H^1(0,1)$ is the unique solution of \eqref{HL}, then the solutions $w^\epsilon$ of \eqref{P} satisfies 
\begin{equation} \label{obs1}
||| w^\epsilon - w_0 |||_{L^2(R^\epsilon)} \to 0, \quad \textrm{ as } \epsilon \to 0.
\end{equation}

In the present paper we introduce a suitable corrector function $\kappa^\epsilon \in H^1(R^\epsilon)$, with $\kappa^\epsilon=\text{o}(\epsilon)$ in $L^2(R^\epsilon)$ 
insomuch  
\begin{equation*} \label{SC}
|||w^\epsilon - w_0 - \kappa^\epsilon|||_{H^1(R^\epsilon)} \to 0, \quad \textrm{ as } \epsilon \to 0. 
\end{equation*}

\begin{remark}
i) If the function $ r(x)$ is continuous, the integral formulation \eqref{HL}
is the weak form of the homogenized equation
\begin{equation} \label{HEq}
\left\{
\begin{gathered}
-\frac{1}{p(x)}(r(x)w_{x})_x + w =  f(x), \quad x \in (0,1)\\
w_x(0)=w_x(1)=0
\end{gathered}
\right.
\end{equation} 
with $f(x)=\hat f(x)/p(x)$.
\par\noindent ii) If initially the function $f^\epsilon(x,y)=f_0(x)$, then it is not difficult to see that $\hat f^\epsilon(x)=(G_\epsilon(x)+b(x))f_0(x)$ and $ G_\epsilon(x) + b(x) \rightharpoonup \frac{|Y^*(x)|}{L} = p(x)$, $w-L^2(0,1)$, as $\epsilon\to 0$,  and therefore, $\hat f(x)=p(x)f_0(x)$. 
\end{remark}


\subsection{Smoothness of $X$} \label{XTHETA}

The auxiliary functions $X(x)$ introduced by the problem \eqref{AUXP} are originally defined in the representative cell $Y^*(x)$, that depends on $x \in (0,1)$. Such variable performs a perturbation on $Y^*(x)$ changing its height at the direction $z$ which also has influence on the auxiliary solution $X(x)$.  Since such a \emph{domain perturbation} is smooth so is $X(x) \in H^1(Y^*)$ with respect to $x$. 
In fact we are able to compute at least one derivative with respect to $x\in (0,1)$ by using Corollary \ref{PAPEN2} proved in Appendix \ref{appendix}. There we combine methods from \cite[Proposition A.2]{AP2} and \cite[Chapter 2]{Henry} where more general results can be found.
This is the reason of our smoothness assumptions to the functions $b$ and $G$ in $(H1)$ and $(H2)$.

Thus, in order to consider $X(x)$ as a function in the thin domain $R^\epsilon$, with the purpose to define the corrector term $\kappa$, we first use the periodicity at variable $y$ to extend them to the band
$$
Y(x) = \{ (y,z) \in \R^2 \, : \, y \in \R, \, -b(x) < z < G(x,y) \}, \quad \textrm{ for all } x \in (0,1).
$$
Then, we properly compose them with the diffeomorphisms
\begin{equation} \label{DIFFEO}
T^\epsilon: \R^2 \to \R^2 \, : \, (x_1,x_2) \mapsto (x_1/\epsilon, x_2/\epsilon), \quad \textrm{ for } \epsilon \in (0, \epsilon_0),
\end{equation}
defining, at the end, the function 
\begin{equation}\label{eq:diff-chang}
X(x) \circ T^\epsilon: R^\epsilon \to \R \text{ for all } x \in (0,1).
\end{equation}

Finally, with some notational abuse, we consider the function
\begin{equation} \label{SL0}
 X: (0,1) \times R^\epsilon \mapsto \R, \quad X(y,y,z) = X(y)\circ T^\epsilon \left(y/\epsilon,z/\epsilon\right) 
\end{equation}
correlating variables $x$ and $y$ of the original function $X$ given by \eqref{AUXP}. 
In order to simplify our notation at the analysis below we still denote functions \eqref{SL0} by 
\begin{equation} \label{XonR}
X = X(x_1)\left(x_1/\epsilon,x_2/\epsilon\right) \quad \textrm{ whenever } (x_1,x_2) \in R^\epsilon.
\end{equation}

Now, under these considerations, we can obtain some estimates for $X$ on $R^\epsilon$. Indeed, 
\begin{equation*}
\begin{split}
\| X \|_{L^2(R^\epsilon)}^2  & =  \int_{R^\epsilon} \left| X(x_1)({x_1}/\epsilon,{x_2}/\epsilon)  \right|^2 dx_1 dx_2 \\
& = \epsilon^2 \int_0^{1/\epsilon} \int_{-b(\epsilon y_1)}^{G(\epsilon y_1,y_2)} |X(\epsilon y_1) (y_1,y_2) |^2 \, dy_1 dy_2  \\
& \leq  C \sum_{k=1}^{1/\epsilon L} \epsilon^2 \int_{Y^*(0)} \left| X(0)(y,z)  \right|^2 dy dz  \\
& \leq {\epsilon} \, \hat C
\end{split}
\end{equation*}
for all $\epsilon \in (0,\hat \epsilon)$, for $\hat \epsilon >0$ small enough with $\hat C>0$ (independent of $\epsilon$). This is provided  by the continuity of the application 
$$
[0,1] \ni s \mapsto X(s) \circ T^\epsilon \in H^1(R^\epsilon)
$$ 
at $s=0$, for each $\epsilon > 0$. 
Thus, using the notation introduced at \eqref{XonR}, we have that 
\begin{equation} \label{eq:Xstim}
 \| X \|_{L^2(R^\epsilon)} = O(\sqrt{\epsilon}). 
\end{equation}

\begin{remark}
Analogously, we can introduce similar notations for the derivatives of $X$ with respect to variables $x$, $y$ and $z$, obtaining that the norms $\| \partial_x X \|_{L^2(R^\epsilon)}$, $\| \partial_y X \|_{L^2(R^\epsilon)}$, and $\| \partial_z X \|_{L^2(R^\epsilon)}$  satisfy \eqref{eq:Xstim} for all $\epsilon \in (0,\hat \epsilon)$.
\end{remark}

\section{First-order Corrector}\label{SFC}

As we mentioned before the solutions $w^\epsilon$ of \eqref{P} doesn't converge in general in $H^1$-norms. Therefore, following Bensoussan, Lions and Papanicolaou in \cite{BLP}, we use the auxiliary solution $X$ to introduce the first-order corrector term
\begin{equation}\label{FOC1}
\kappa^\epsilon(x_1,x_2) = - \epsilon X(x_1)\left(\dfrac{x_1}{\epsilon},\dfrac{x_2}{\epsilon}\right) \dfrac{dw_0}{dx_1}(x_1),
\quad (x_1,x_2) \in R^\epsilon.
\end{equation}

\begin{remark} Let us still observe that the function $\kappa^\epsilon$ defined in \eqref{FOC1} is not the standard corrector according to \cite{BLP, CD, PS} since its dependence with respect to variable $x_1$ involves the changeable auxiliary solution $X$ with respect to the representative cell  $Y^*(x)$ whose smoothness has been discussed in Section $\ref{XTHETA}$.
As a matter of fact, we believe that this is the main contribution of our work established here. 
\end{remark}

Now, we have elements to proof the following result

\begin{theorem}\label{theo:first-order}
Let $w^\epsilon$ be the solution of problem $\eqref{P}$ with $f^\epsilon \in L^2(R^\epsilon)$ satisfying
$$
||| f^\epsilon |||_{L^2(R^\epsilon)} \leq C
$$
for some $C > 0$ independent of $\epsilon$.
Consider the family of functions $\hat f^\epsilon \in L^2(0,1)$ defined by
\begin{equation}\label{def-hat-f}
\hat f^\epsilon (x_1) = \epsilon^{-1} \int_{-\epsilon b(x_1)}^{\epsilon G_\epsilon(x_1)} f^\epsilon(x_1,x_2) \, dx_2 
\end{equation}
with $G_\epsilon(\cdot)$ and $b(\cdot)$ satisfying hypotheses $(H1)$ and $(H2)$.

If $\hat f^\epsilon \rightharpoonup \hat f $\,  w-$L^2(0,1)$, then
\begin{equation}\label{eq:conv-f-corr}
\lim_{\epsilon \to 0} ||| w^\epsilon  - w_0 - \kappa^\epsilon |||_{H^1(R^\epsilon)} = 0,
\end{equation}
where $\kappa^\epsilon$ is the first-order corrector of $w^\epsilon$ defined in \eqref{FOC1}, and $w_0 \in H^2(0,1) \cap {C}^1(0,1)$ is the  unique solution of the homogenized equation $\eqref{HEq} $
with
\begin{equation} \label{F0}
f(x) = \hat f(x)/p(x).
\end{equation}
\end{theorem}

\begin{proof}
We have $a_\epsilon(\varphi, w^\epsilon) = (\varphi,f^\epsilon)_\epsilon$,  for all $\varphi \in H^1(R^\epsilon)$.
Taking $w_0 + \kappa^\epsilon \in H^1(R^\epsilon)$\footnote{Note that here $w_0$ is considered to be a function of both variables $x_1$ and $x_2$, simply by $w_0(x_1,x_2)=w_0(x_1)$.} as a test function, we obtain that
\begin{equation}\label{eq:nor-firs-corr}
\begin{split}
|||w^\epsilon - w_0 - \kappa^\epsilon|||^2_{H^1(R^\epsilon)} & =
a_\epsilon(w^\epsilon - w_0 - \kappa^\epsilon,w^\epsilon - w_0 - \kappa^\epsilon) \\
& =  a_\epsilon(w^\epsilon,w^\epsilon - w_0 - \kappa^\epsilon) - a_\epsilon(w_0 + \kappa^\epsilon,w^\epsilon) \\
& + a_\epsilon(w_0 + \kappa^\epsilon,w_0 + \kappa^\epsilon)\\
& = (w^\epsilon-2(w_0 + \kappa^\epsilon) , f^\epsilon)_\epsilon + a_\epsilon(w_0 + \kappa^\epsilon,w_0 + \kappa^\epsilon) .
\end{split}
\end{equation}
Now, performing the change of variables $(x,y) \to (x,y/\epsilon)$ on problem \eqref{P}, it follows from \cite[Theorem 2.3]{AP2}, as observed in \eqref{obs1} that
\begin{equation*}
\epsilon^{-1} \| w^\epsilon - w_0 \|_{L^2(R^\epsilon)} \to 0 \textrm{ as } \epsilon \to 0,
\end{equation*}
i.e., $||| w^\epsilon - w_0 |||_{L^2(R^\epsilon)} \stackrel{\epsilon \to 0} \longrightarrow 0$. Hence
\begin{equation} \label{*}
(w^\epsilon-w_0,f^\epsilon)_\epsilon \le  |||w^\epsilon-w_0|||_{L^2(R^\epsilon)} |||f^\epsilon |||_{L^2(R^\epsilon)}  \to 0, \quad \text{as } \epsilon \to 0.
\end{equation}
From estimate \eqref{eq:Xstim}, we also obtain (uniformly for $x \in (0,1)$) that 
\begin{eqnarray} \label{**}
(\kappa^\epsilon,f^\epsilon)_\epsilon
& \le & \epsilon^{-1} \|\kappa^\epsilon \|_{L^2(R^\epsilon)} \|f^\epsilon \|_{L^2(R^\epsilon)} \nonumber \\
& \le & \| X(x) \|_{L^2(R^\epsilon)} \|f^\epsilon \|_{L^2(R^\epsilon)} \left\| dw_0/dx \right\|_{L^\infty(0,1)} \\
& \le & \sqrt{\epsilon} \, \hat C \to 0, \quad \textrm{ as } \epsilon \to 0. \nonumber
\end{eqnarray}
Moreover, from the convergence $\hat f^\epsilon \rightharpoonup \hat f $\,  w-$L^2(0,1)$, we obtain that
\begin{eqnarray} \label{***}
(w_0,f^\epsilon)_\epsilon & = & \epsilon^{-1} \int_0^1 w_0 (x_1) \int_{-\epsilon b(x_1)}^{\epsilon G(x_1,x_1/\epsilon) } f^\epsilon (x_1,x_2) \, dx_2dx_1 \nonumber \\
& = & \int_0^1 w_0 (x) \hat f^\epsilon(x)\, dx \nonumber \\
& \to & \int_0^1 w_0 (x) \hat f (x) \, dx, \quad \text{ as } \epsilon \to 0 .
\end{eqnarray}
Therefore, we get from \eqref{F0}, \eqref{*}, \eqref{**}, \eqref{***} and \eqref{IP0} that
\begin{equation} \label{eq:TCH}
(w^\epsilon-2(w_0 + \kappa^\epsilon) , f^\epsilon)_\epsilon \stackrel{\epsilon \to 0} \longrightarrow ( w_0, f )_0.
\end{equation}

Next, we will show that
\begin{equation} \label{eqeq}
a_\epsilon(w_0 + \kappa^\epsilon,w_0 + \kappa^\epsilon) \to  a_0(w_0 ,w_0 ), \quad \textrm{ as } \epsilon \to 0.
\end{equation}
In order to do that, we need to pass to the limit on the expression 
\begin{eqnarray} \label{eq:ARGUE}
& & a_\epsilon(w_0 + \kappa^\epsilon,w_0 + \kappa^\epsilon)  = \epsilon^{-1} \int_{R^\epsilon}\big\{ |\nabla(w_0 + \kappa^\epsilon)|^2 + |w_0 + \kappa^\epsilon|^2 \big\} dx_1 dx_2 \nonumber \\
 &  & \qquad =\epsilon^{-1} \int_{R^\epsilon} \left\{ \dfrac{dw_0}{dx_1}^2 \left( \left( 1- \partial_{y} X \right)^2 + {\partial_{z} X}^2 \right) + w_0^2 \right\}  dx_1 dx_2 \label{aI1} \\
& & \qquad - 2 \int_{R^\epsilon} \left\{ \left(1 - \partial_y X \right) \left( \dfrac{dw_0}{dx_1}^2 \, \partial_x X + \dfrac{dw_0}{dx_1} \, \dfrac{d^2w_0}{{dx_1}^2} \, X \right) + w_0 \, \dfrac{dw_0}{dx_1} \, X \right \} dx_1 dx_2  \label{aI2} \\
& & \qquad + \epsilon \int_{R^\epsilon} \left\{ \dfrac{dw_0}{dx_1}^2 \, {\partial_x X}^2 + 2 \, \dfrac{dw_0}{dx_1} \, \dfrac{d^2w_0}{{dx_1}^2} \, X \, \partial_x X + X^2 \left( \dfrac{d^2w_0}{{dx_1}^2}^2 + \dfrac{dw_0}{dx_1}^2 \right)  \right\} dx_1 dx_2 \label{aI3}
\end{eqnarray}
when $\epsilon$ goes to zero.
By estimates in \eqref{eq:Xstim}, we shall have to consider the parcel
\begin{eqnarray}
 & & \epsilon^{-1} \int_{R^\epsilon} \left\{ \dfrac{dw_0}{dx_1}^2 \left( \left( 1 - {\partial_y X}(x_1) \left(\frac{x_1}{\epsilon},\frac{x_2}{\epsilon} \right) \right)^2 + {\partial_z X(x_1) \left(\frac{x_1}{\epsilon},\frac{x_2}{\epsilon} \right) }^2 \right) + {w_0(x_1)}^2 \right\} \, dx_1 dx_2 \nonumber \\
& & =  \int_0^1 \int_{-b(x)}^{G(x,x/\epsilon)} \left\{ \dfrac{dw_0}{dx}^2 \left( \left( 1 - {\partial_y X}(x) \left(\frac{x}{\epsilon},z \right) \right)^2 + {\partial_z X(x)\left(\frac{x}{\epsilon},z \right) }^2 \right) + {w_0(x)}^2 \right\} dz  dx \nonumber \\
& & =  \int_0^1 \dfrac{dw_0}{dx}^2  \int_{-b(x)}^{G(x,x/\epsilon)} \left\{ \left( 1 - \,  {\partial_y X}(x) \left(\frac{x}{\epsilon},z \right) \right)^2 + {\partial_z X(x)\left(\frac{x}{\epsilon},z \right) }^2 \right\} dz dx\label{aI11} \\
& & \qquad \quad 
+ \int_0^1  {w_0(x)}^2 \left( G(x,x/\epsilon) + b(x) \right) \, dx.  \label{aI12}
\end{eqnarray}
Since the function $G(x,y)$ is $L$-periodic at variable $y$, we also have that
$$
\Phi(y) = \int_{-b(x)}^{G(x,y)} \left\{ \left( 1 - \,  {\partial_y X}(x) \left(y,z \right) \right)^2 + {\partial_z X(x)\left(y,z \right) }^2 \right\} dz
$$
is $L$-periodic for each $x \in (0,1)$ fixed. Consequently, we obtain from \cite[Theorem 2.6]{CD}\footnote{This theorem is often called \emph{Average Convergence Theorem for Periodic Functions}. See also \cite{CP}.} that 
$$
\Phi(x/\epsilon) \rightharpoonup 1/L \int_0^L \Phi(s) \, ds, \ w^*-L^\infty(0,1),
$$ 
as $\epsilon \to 0$. Therefore
\begin{eqnarray} \label{eqConv1}
& & \epsilon^{-1} \int_{R^\epsilon}  \dfrac{dw_0}{dx_1}^2 \left\{ \left( 1 - {\partial_y X}(x_1) \left(\frac{x_1}{\epsilon},\frac{x_2}{\epsilon} \right) \right)^2 + {\partial_z X(x_1) \left(\frac{x_1}{\epsilon},\frac{x_2}{\epsilon} \right) }^2 \right\} \, dx_1 dx_2 \nonumber \\
& & \qquad \to \int_0^1 \dfrac{dw_0}{dx}^2 \, \frac{1}{L} \int_{Y^*(x)} \left\{ \left( 1 - {\partial_y X}(x)\left( y, z \right) \right)^2 + {\partial_z X(x) \left( y, z \right) }^2 \right\} \, dy dz dx
\end{eqnarray}
as $\epsilon \to 0$. See also \cite[Lemma 4.2]{AP3}.

Now, since $X(x)$ satisfies \eqref{AUXP}, we have that
$$
\int_{Y^*(x)} | \nabla_{y,z} X(x)|^2 dydz = \int_{B_1} N_1 \, X(x) \, dS
$$ 
where $N=(N_1,N_2)$ is the unit outward normal to $B_1$, the upper boundary of $\partial Y^*(x)$. Hence, we obtain from 
$$
\int_{B_1} N_1 \, X(x) \, dS 
= \int_{Y^*(x)} {\rm{ div }_{y,z}} 
\left(
\begin{array}{c}
X(x) \\ 0
\end{array} 
\right)
\, dydz 
= \int_{Y^*(x)} \partial_y X(x) \, dydz
$$ 
that, for all $x \in (0,1)$, 
\begin{equation} \label{eqSL1}
\int_{Y^*(x)} | \nabla_{y,z} X(x)|^2 dydz =  \int_{Y^*(x)} \partial_y X(x) \, dydz.
\end{equation}
Thus, due to \eqref{eqSL1}, we have  
\begin{eqnarray}
& & \int_{Y^*(x)} \left\{ \left( 1 - {\partial_y X}(x) \left( y, z \right) \right)^2 + {\partial_z X(x) \left( y, z \right) }^2 \right\} \, dy dz \nonumber \\
& & = \int_{Y^*(x)} \left\{ 1 - 2 \, {\partial_y X}(x) \left( y, z \right) + | \nabla_{y,z} X(x)|^2 \right\} \, dy dz  \nonumber \\
& & = \int_{Y^*(x)} \left\{ 1 - {\partial_y X}(x) \left( y, z \right) \right\} \, dy dz \nonumber \\
& & = r(x) \, L \label{eqSL2},
\end{eqnarray}
where the positive function $r(x)$, called \emph{homogenization coefficient}, is introduced at \eqref{HC}.

Then, putting \eqref{eqConv1} and \eqref{eqSL2} together, we get the following limit
\begin{eqnarray} \label{eq*}
& & \epsilon^{-1} \int_{R^\epsilon}  \dfrac{dw_0}{dx_1}^2 \left\{ \left( 1 - {\partial_y X}(x_1) \left(\frac{x_1}{\epsilon},\frac{x_2}{\epsilon} \right) \right)^2 +   {\partial_z X(x_1) \left(\frac{x_1}{\epsilon},\frac{x_2}{\epsilon} \right) }^2  \right\} \, dx_1 dx_2 \nonumber \\
& & \qquad \to \int_0^1 r(x) \, \dfrac{dw_0}{dx}^2 \, dx, \quad \textrm{ as } \epsilon \to 0.
\end{eqnarray}
Back to the parcel \eqref{aI12}, we notice that as in \cite[Theorem 2.6]{CD} and \cite[Lemma 4.2]{AP3} we have that 
\begin{equation} \label{eq**}
\int_0^1  {w_0(x)}^2 \left( G(x,x/\epsilon) + b(x) \right) \, dx   \to \int_0^1 p(x) \, {w_0(x)}^2 \, dx, \quad \textrm{ as } \epsilon \to 0,
\end{equation}
where the function $p$ is defined at \eqref{HC}, and it is related to the measure of the representative cell $Y^*(x)$.

Therefore, the limit \eqref{eqeq} follows from \eqref{eq*} and \eqref{eq**}.

Finally, according to \eqref{eq:nor-firs-corr}, we complete the proof getting
$$
 |||w^\epsilon - w_0 - \kappa^\epsilon |||^2_{H^1(R^\epsilon)}  \stackrel{\epsilon \to 0}{\longrightarrow} a_0(w_0,w_0) - (w_0 , f)_0 = 0,
$$
from \eqref{eqeq}, \eqref{eq:TCH}, and using that $w_0$ satisfies \eqref{HEq}.

\end{proof}

\appendix
\section{ A perturbation result in the basic cell.}
\label{appendix}

In the proof of the main results in Section \ref{SFC} we have used the smoothness of the auxiliary solution $X$ given by \eqref{AUXP} with respect to variable $x \in (0,1)$ evaluating the function $\partial_x X$.  In order to get such a derivative we analyze here, without loss of generality, how the function $X \in H^1(Y^*(x))$, solution of
\begin{equation} \label{AUXKato1}
\left\{
\begin{gathered}
- \Delta X(x)  =  0  \textrm{ in } Y^*(x)  \\
\frac{\partial X(x)}{\partial N}(y,g(y))= N_1(x) \quad  \textrm{ on } B_1(x)\\
\frac{\partial X(x)}{\partial N}(y,z)=0 \quad \textrm{ on } B_2(x)\\
 X(x)(\cdot, z) \quad L\text{ - periodic} \\ 
\int_{Y^*(x)} X(x) \; dy dz  =  0  
\end{gathered}
\right.
\end{equation}
in the representative cell $Y^*(x)$ defined by
\begin{equation} \label{BCEL}
Y^*(x) = \{ (y,z) \in \R^2 \ : \ 0< y < L,  0 < z < G(x,y) \},
\end{equation}
depends on the function $G$ where $B_1(x)$ and $B_2(x)$ are respectively the upper and lower boundary of $\partial Y^*(x)$ for $x \in (0,1)$, and $N(x)=(N_1(x),N_2(x))$ is the unitary outside vector field to $\partial Y^*(x)$. 

For this we consider the following class of admissible functions
\begin{equation}\label{admissible}
A(M)=\{G\in C^1(\R), L-\hbox{periodic},\, 0<G_0\leq G(\cdot)\leq G_1, |G'(s)|\leq M\}
\end{equation}
endowed with the norm 
$
\| G \| = \max_{0\leq i \leq 1} \sup_{\xi \in \R} |\partial_\xi^i G(\xi)|.
$ 
We call by $Y^*(G)$ the basic cell (\ref{BCEL}) and by $X(G)$ the unique solution of problem (\ref{AUXKato1}) for $G$ varying in $A(M)$.

Following the approach applied in \cite{A, HR, Henry} we begin by making the transformation on the cell $Y^*(G)$ into the domain $Y^*(\hat G)$ for each $\hat G \in A(M)$
\begin{eqnarray*}
L_{\hat G} & : & Y^*(G) \mapsto Y^*(\hat G) \\
& & (z_1, z_2) \to (z_1, \hat F(z_1) \; z_2) = (y_1, y_2)
\end{eqnarray*}
where $\hat F(z) = \frac{\hat G(z)}{G(z)}$.
The Jacobian matrix for this transformation is
$$
J L_{\hat G} (z_1,z_2) = \left(
\begin{array}{cc}
1 & 0 \\
\hat F'(z_1) z_2 & \hat F(z_1)
\end{array}
\right)
$$
whose determinant is given by 
\begin{equation} \label{RF0}
|J L_{\hat G} (z_1,z_2)|=\hat F(z_1)=\frac{\hat G(z_1)}{G(z_1)}
\end{equation}
that is a rational function with respect to $G$ and $\hat G$ satisfying 
\begin{equation} \label{RF00}
0< \frac{G_0}{G_1} \leq \hat F(z) \leq \frac{G_1}{G_0}, \quad \textrm{ for all } z \in \R.
\end{equation} 

Consequently, using $L_{\hat G}$ we have that problem (\ref{AUXKato1}) in the cell $Y^*(\hat G)$ is equivalent to
\begin{equation} \label{AUXKato00}
\left\{
\begin{array}{ll}
- \frac{1}{\hat F} {\rm div }( B_{\hat G} W ) = 0 & \textrm{ in } Y^*(G) \\
B_{\hat G} W \cdot N = 0 & \textrm{ on } B_2(G) \\
B_{\hat G} W \cdot N = - \frac{\partial_2 \hat G}{\sqrt{1 + \left( \partial_2 \hat G \right)^2}} & \textrm{ on } B_1(G) \\
W(\cdot,z_2) = W(\cdot,z_2) & \textrm{ $L$-periodic } \\
\int_{Y^*(G)} W \, |J L_{\hat G}| \, dz_1 dz_2 = 0
\end{array}
\right. 
\end{equation}
where 
\begin{equation} \label{RF1}
(B_{\hat G} W)(z_1,z_2) = \left(
\begin{array}{c}
\hat F(z_1) \; \frac{\partial W}{\partial z_1}(z_1,z_2) - \hat F'(z_1) \; z_2 \; \frac{\partial W}{\partial z_2}(z_1,z_2) \\
- \hat F'(z_1) \; z_2 \; \frac{\partial W}{\partial z_1}(z_1,z_2) 
+ \frac{1}{\hat F(z_1)} \; ( 1 + (z_2 \; \hat F'(z_1))^2 ) \; \frac{\partial W}{\partial z_2}(z_1,z_2)
\end{array}
\right).
\end{equation}

More precisely, we have that $X(\hat G)$ is the solution of (\ref{AUXKato1}) in $Y^*(\hat G)$ 
if and only if $W(\hat G) = X(\hat G) \circ L_{\hat G}$ satisfies equation (\ref{AUXKato00}) in $Y^*(G)$.
Further, if we define $$U(\hat G)=W(\hat G)-\frac{1}{|Y^*(G)|}\int_{Y^*(G)}W(\hat G)$$
then, $U(\hat G)$ is the unique solution of the problem
\begin{equation} \label{AUXKato0}
\left\{
\begin{array}{ll}
- \frac{1}{\hat F} {\rm div }( B_{\hat G} U) = 0 & \textrm{ in } Y^*(G) \\
B_{\hat G} U \cdot N = 0 & \textrm{ on } B_2(G) \\
B_{\hat G} U \cdot N = - \frac{\partial_2 \hat G}{\sqrt{1 + \left( \partial_2 \hat G \right)^2}} & \textrm{ on } B_1(G) \\
U(\cdot,z_2) = U(\cdot,z_2) & \textrm{ $L$-periodic } \\
\int_{Y^*(G)} U  \, dz_1 dz_2 = 0
\end{array}
\right. .
\end{equation}

We also have that, for each $\hat G \in A(M)$, the variational formulation of (\ref{AUXKato0}) is given by the bilinear form 
\begin{equation} \label{BForm}
\begin{array}{rl} 
\rho_{\hat G}: H(G) \times H(G) &\mapsto \R \\
 \displaystyle (U,V) &\to \displaystyle  \int_{Y^*(G)} B_{\hat G} U \cdot \nabla \left( \frac{V}{\hat F} \right) \; dz_1 dz_2
\end{array}
\end{equation}
where $H(G)$ is the Hilbert space defined by 
\begin{eqnarray*}
H(G) & = & \left\{ U \in H^1(Y^*(G)) \, : \, U(\cdot,z_2) = U(\cdot,z_2) \textrm{ is $L$-periodic }, \, \int_{Y^*(G)}U=0 \right\}
\end{eqnarray*}
and endowed with the $H^1(Y^*(G))$ norm. 
In fact, $\rho_{\hat G}$ is a coercive bilinear form since we have imposed the condition $\int_{Y^*(G)} U = 0$ on the space $H(G)$.

Therefore, $X(\hat G)$ is the solution of (\ref{AUXKato1}) in $Y^*(\hat G)$ if and only if 
\begin{equation}\label{def-of-U}
U(\hat G) = X(\hat G) \circ L_{\hat G}-\frac{1}{|Y^*(G)|}\int_{Y^*(G)}X(\hat G) \circ L_{\hat G}
\end{equation}
satisfies 
\begin{equation}\label{weak-form-hat G}
\rho_{\hat G}(U,V) = 
- \int_{B_1(G)} \frac{\hat G'}{\sqrt{1 + ( \hat G ')^2}} \; \frac{V}{\hat F} \; d\mathcal{S},  \quad
\textrm{ for all } V \in H(G).
\end{equation}
In particular the coercive bilinear form associated to (\ref{AUXKato1}) in $Y^*(G)$ is given by 
\begin{equation}
\begin{array}{rl}
\rho_{G}: H(G) \times H(G) &\mapsto \R \\
 \displaystyle (U,V) &\to \displaystyle  \int_{Y^*(G)} \nabla  U \cdot \nabla V  \; dz_1 dz_2 \nonumber 
\end{array}
\end{equation}
in which the weak formulation is  
$$
\rho_{G}(U,V) = 
- \int_{B_1(G)} \frac{G'}{\sqrt{1 + ( G ')^2}} \; V \; d\mathcal{S}, \quad 
\textrm{ for all } V \in H(G).
$$

Thus we have performed the boundary perturbation problem given by \eqref{AUXKato1}, where the function spaces are varying with respect to $\hat G \in A(M)$, into a problem of perturbation of linear operators defined by the bilinear form \eqref{BForm} and condition \eqref{weak-form-hat G} in the fixed space $H(G)$. 
%

\begin{remark} \label{FD} 
In order to obtain the smoothness of auxiliary solutions $X$ with respect to $\hat G$, we first observe that the maps $\rho$ and $\gamma$ defined by 
\begin{equation} \label{b0G}
\begin{array}{rl}
\rho: A(M) \times H(G) \times H(G) &\mapsto \R \\
 \displaystyle (\hat G, U,V) &\to \rho_{\hat G}(U,V) 
\end{array}
\end{equation}
and 
\begin{equation} \label{b0G2}
\begin{array}{rl}
\gamma: A(M) \times H(G) &\mapsto \R \\
 \displaystyle (\hat G, V) &\to \int_{B_1(G)} \frac{\hat G'}{\sqrt{1 + ( \hat G ')^2}} \; \frac{V}{\hat F} \; d\mathcal{S} 
\end{array}
\end{equation}
are Fr\'echet differentiable in the Banach spaces $A(M) \times H(G) \times H(G)$ and $A(M) \times H(G)$.
Indeed, setting $M>0$ and $G \in A(M)$, we have for each $\hat G \in H(G)$ that $\rho(\hat G, \cdot, \cdot)$ is a coercive bilinear form in $H(G) \times H(G)$, and $\gamma(\hat G, \cdot)$ is a linear functional on $H(G)$. Moreover, due to expression \eqref{RF0}, \eqref{RF00} and \eqref{RF1}, we have that $\rho$ and $\gamma$ are rational functions with respect to $\hat G \in A(M)$ and its derivatives. Therefore, it follows from \cite{D} that $\rho$ and $\gamma$ defined by \eqref{b0G} and \eqref{b0G2} are analytic maps in $A(M) \times H(G) \times H(G)$ and $A(M) \times H(G)$ respectively (see \cite[Cchapter 2]{Henry}).
\end{remark}

\begin{proposition}\label{PAPEN}
Let us consider the family of solutions $U(\hat G)$ of problem \eqref{AUXKato0} for each $\hat G \in A(M)$ where $A(M)$ is the functional space defined in (\ref{admissible}) for some constant $M>0$. 

Then, the map $\hat G \to U(\hat G) \in H(G)$ is analytic in $A(M)$.
\end{proposition}

\begin{proof}
Using the applications $\rho(\hat G, \cdot,\cdot)$ is a coercive bilinear form, and $\gamma(\hat G, \cdot)$ is a linear functional, we can define the linear operator $Z(\hat G): H(G) \mapsto H(G)'$, and the linear functional $S(\hat G): H(G) \mapsto \R$ given by $<Z(\hat G) \phi, \varphi> = \rho(\hat G, \phi, \varphi)$ and  $S(\hat G) \varphi = \gamma(\hat G, \varphi)$ such that 
$$
Z(\hat G) U(\hat G) + S(\hat G) = 0, \quad  \textrm{ for all } \hat G \in A(M)
$$ 
where $H(G)'$ is the dual space of $H(G)$ and $U(\hat G)$ is the unique solution of problem \eqref{AUXKato0} for each $\hat G \in A(M)$ that satisfies expression \eqref{weak-form-hat G}.

Hence, due to Remark \ref{FD}, we have that the desired result is a direct consequence of the Implicit Functional Theorem applied to the map
$$
\begin{array}{rl}
\Phi: A(M) \times H(G) &\mapsto H(G)' \\
 \displaystyle (\hat G, U) &\to Z(\hat G) \, U+ S(\hat G).
\end{array}
$$

\end{proof}

Consequently we have the following result:

\begin{corollary}\label{PAPEN2}
Let $T^\epsilon: \R^2 \mapsto \R^2$ be the diffeomorphism defined in \eqref{DIFFEO} for each $\epsilon > 0$, and consider the family of solutions $X(\hat G)$ given by the problem \eqref{AUXKato1} for $\hat G \in A(M)$ where $A(M)$ is the functional space defined in (\ref{admissible}) for some constant $M>0$. 

Then, the map $\hat G \to X(\hat G) \circ T^\epsilon \in H^1(R^\epsilon)$ is analytic in $A(M)$.
\end{corollary}


\end{document}